\documentclass[ECP]{ejpecp} 

\usepackage[inline]{enumitem}   
\usepackage{mathrsfs}		
\usepackage{xfrac}              
\usepackage[round,comma,authoryear,sort&compress]{natbib} 

\SHORTTITLE{Local martingales in discrete time}

\TITLE{Local martingales in discrete time
}

\AUTHORS{Vilmos Prokaj
  \footnote{Department of Probability Theory and Statistics,
    E\"otv\"os Lor\'and University,
    Budapest \EMAIL{prokaj@cs.elte.hu}}
  \and
  Johannes Ruf
  \footnote{Department of Mathematics,
    London School of Economics and Political Science
    \EMAIL{j.ruf@lse.ac.uk}}
}


\AMSSUBJ{60G42;60G48} 

\SUBMITTED{September 20, 2017} 
\ACCEPTED{April 22, 2018} 

\ARXIVID{1701.04025} 

\VOLUME{0}
\YEAR{2018}
\PAPERNUM{0}
\DOI{10.1214/YY-TN}

\ABSTRACT{For any discrete-time $\P$--local martingale $S$
  there exists a probability measure $\Qu \sim \P$ such that $S$ is a
  $\Qu$--martingale. A new proof for this result is provided. The core idea relies on 
  an appropriate modification of an argument by Chris Rogers, used to
  prove a version of the fundamental theorem  
  of asset pricing in discrete time. This
  proof also yields that, for any $\varepsilon>0$, the measure $\Qu$
  can be chosen so that $\sfrac{\d \Qu}{\d \P} \leq 1+\varepsilon$.  
}

\newcommand{\N}{\mathbb{N}}
\newcommand{\Q}{\mathbb{Q}}
\newcommand{\R}{\mathbb{R}}

\DeclareMathOperator{\interior}{int}

\newcommand{\Lp}[1]{\ensuremath{\textsf{\upshape L}^{#1}}}

\renewcommand{\d}{\mathrm{d}}
\newcommand{\e}{\mathrm{e}}
\newcommand{\E}{\textsf{\upshape E}}
\renewcommand{\P}{\textsf{\upshape P}}
\newcommand{\Qu}{\textsf{\upshape Q}}
\newcommand{\indicator}[1]{\mathbf{1}_{#1}}
\newcommand{\sigalg}[1]{\mathscr{#1}}
\DeclarePairedDelimiter\zjel{(}{)}
\DeclarePairedDelimiter\abs{|}{|}
\DeclarePairedDelimiterX\br[1]{[}{]}{#1}
\DeclarePairedDelimiter\smallset{\{}{\}}
\let\event=\smallset
\DeclarePairedDelimiterX\set[2]\{\}{#1\::\:#2}
\newcommand\given{\nonscript\:\delimsize\vert\nonscript\:\mathopen{}}

\bibpunct{(}{)}{,}{a}{}{,}

\begin{document}

\section{Introduction and related literature}
Let $(\Omega, \sigalg F, \P)$ denote a probability space equipped with
a discrete-time filtration $(\sigalg F_t)_{t \in \N_0}$, where
$\sigalg F_t \subset \sigalg F$. Moreover, let $S = (S_t)_{t \in
  \N_0}$ denote a $d$-dimensional $\P$--local martingale, where $d \in
\N$. Then there exists a probability measure $\Qu$, equivalent to
$\P$, such that $S$ is a $\Qu$--martingale. 
This follows from more general results that relate appropriate
no-arbitrage conditions to the existence of an equivalent martingale
measure;  see  \cite{Dalang:Morton:Willinger} and
\cite{Schachermayer:1992} for the finite-horizon case and  
\cite{Schachermayer:1994} for the infinite-horizon case. These results
are sometimes baptized fundamental theorems of asset pricing. 

More recently, \cite{Kabanov:2008} and \cite{Prokaj:Rasonyi:2010} have
provided a direct proof for the existence of such a measure $\Qu$; see
also Section~2 in \cite{Kabanov:Safarian}. The proof in
\cite{Kabanov:2008}  relies on deep functional analytic results, e.g.,
the Krein-\v Smulian theorem. The proof in \cite{Prokaj:Rasonyi:2010}
avoids functional analysis but requires non-trivial measurable
selection techniques.  

As this note demonstrates, in one dimension, an important but special
case, the Radon-Nikodym derivative $Z_\infty = \sfrac{\d \Qu}{\d \P}$
can be explicitly constructed.  Moreover, in higher dimensions, the
measurable selection results can be simplified. This is done here by
appropriately modifying an ingenious idea of \cite{Rogers:1994}. 

More precisely, the following theorem will be proved in Section~\ref{S:proof}.

\begin{theorem}\label{T:1}
  For all $\varepsilon > 0$, there exists a uniformly integrable
  $\P$--martingale $Z= (Z_t)_{t \in \N_0}$, bounded from above by
  $1+\varepsilon$, with $Z_\infty = \lim_{t \uparrow \infty} Z_t > 0$,
  such that $ZS$ is a $\P$--martingale and such that $\E_\P[Z_t
  |S_t|^p] < \infty$ for all $t \in \N_0$ and $p \in \N$.  
\end{theorem}

The fact that the bound on $Z$ can be chosen arbitrarily close to $1$
seems to be a novel observation.  Considering a standard random walk
$S$ directly yields that there is no hope for a stronger version of
Theorem~\ref{T:1} which would assert that $ZS$ is not only a
$\P$--martingale but also a $\P$--uniformly integrable martingale. 

A similar version of the following corollary is formulated in
\cite{Prokaj:Rasonyi:2010}; it would also be a direct  consequence of
\cite{Kabanov:Stricker:2001}.  To state it, let us introduce the total
variation norm $\| \cdot\|$ for two equivalent probability measures
$\Qu_1, \Qu_2$ as 
\begin{align*}
  \|\Qu_1 - \Qu_2\| = \E_{\Qu_1} \left[\left| \sfrac{\d \Qu_2}{\d \Qu_1} - 1 \right| \right].
\end{align*}

\begin{corollary}\label{C:1}
  For all $\varepsilon > 0$, there exists a probability measure $\Qu$,
  equivalent to $\P$, such that $S$ is a $\Qu$--martingale, $\|\P -
  \Qu\| < \varepsilon$, and   $\E_\Qu[|S_t|^p] < \infty$ for all $t
  \in \N_0$ and $p \in \N$.  
\end{corollary}

To reformulate Corollary~\ref{C:1} in more abstract terms, let us
introduce the spaces 
\begin{align*}
  \mathcal Q_l  &= \left\{\Qu \sim \P:\, \text{$S$ is a $\Qu$--local martingale}\right\};\\
 \mathcal Q^{p} &= \left\{\Qu \sim \P:\, \text{$S$ is a $\Qu$--martingale with $\E_\Qu[|S_t|^p] < \infty$ for all $t \in \N_0$}\right\}, \qquad p > 0.
\end{align*}
Then Corollary~\ref{C:1} states that the space $\bigcap_{p \in \N}
 \mathcal Q^{p}$ is dense in $ \mathcal Q_l$ with respect to the
total variation norm $\| \cdot \|$.

\begin{proof}[Proof of Corollary~\ref{C:1}]
  Consider the $\P$--uniformly integrable martingale $Z$ of
  Theorem~\ref{T:1}, with $\varepsilon$ replaced by
  $\sfrac{\varepsilon}{2}$.  Then the probability measure
  $\Qu$, given by $\sfrac{\d \Qu }{\d \P} = Z_\infty$, satisfies the
  conditions of the assertion.  Indeed, we only need to observe
  that
  \begin{align*}
    \E_{\P} \br{\abs{Z_\infty - 1}}=2\E_{\P}\br{(Z_\infty
    -1)\indicator{\event{Z_\infty>1}}}\leq \varepsilon,
  \end{align*}
  where we used that $\E_{\P}\br{Z_\infty-1}=0$ and the assertion follows.
\end{proof}

\section{Generalized conditional expectation and local martingales}
For sake of completeness, we review the relevant facts related to
local martingales in discrete time. To start,  
note that for a sigma algebra $\sigalg G \subset \sigalg F$ and a
nonnegative random variable $Y$, not necessarily integrable, we can
define the so called generalized conditional expectation 
\begin{align*}
  \E_\P \br{Y \given \sigalg G} = \lim_{k \uparrow \infty} \E_\P\br{Y \wedge k \given \sigalg G}.
\end{align*}
Next, for a general random variable $W$ with $\E_\P\br{|W| \given \sigalg G} <
\infty$, but not necessarily integrable, we can define the generalized
conditional expectation 
\begin{align*}
  \E_\P\br{W \given \sigalg G} = \E_\P\br{W^+ \given \sigalg G}  - \E_\P\br{W^-\given \sigalg G}.
\end{align*}

For a stopping time $\tau$ and a stochastic process $X$ we write $X^{\tau}$ to denote the process obtained from stopping $X$ at time $\tau$.

\begin{definition}
  A stochastic process $S = (S_t)_{t \in \N_0}$ is 
  \begin{itemize}
  \item a $\P$--martingale if $\E_\P[|S_t|] < \infty$ and
    $\E_\P\br{S_{t+1} \given\sigalg F_t} = S_t$ for all $t \in \N_0$; 
  \item  a $\P$--local martingale if there exists a sequence
    $(\tau_n)_{n \in \N}$ of stopping times such that $\lim_{n
      \uparrow \infty} \tau_n = \infty$ and $S^{\tau_n}
    \indicator{\{\tau_n > 0\}}$ is a $\P$--martingale; 
  \item a $\P$--generalized martingale if $\E_\P\br{|S_{t+1}|\given\sigalg
    F_t} < \infty$ and $\E_\P\br{S_{t+1} \given \sigalg F_t} = S_t$ for all $t
    \in \N_0$. 
  \end{itemize}
\end{definition}

\begin{proposition}\label{P:1}
  Any $\P$--local martingale is a $\P$--generalized martingale.
\end{proposition}
This proposition dates back to Theorem~II.42 in \cite{Meyer:1972}; see
also Theorem~VII.1 in \cite{Shiryaev:1996}.  Its reverse direction
would also be true but  will not be used below. A direct corollary of
the proposition is that a $\P$--local martingale $S$ with $\E_\P[|S_t|] <
\infty$ for all $t \in \N_0$ is indeed a $\P$--martingale. 

For sake of completeness, we will provide a proof of the proposition here.
\begin{proof}[Proof of Proposition~\ref{P:1}]
Let $S$ denote a $\P$--local martingale.
  Fix $t \in \N_0$ and a localization sequence $(\tau_n)_{n \in
    \N}$. For each $n \in \N$, we have, on the event 
  $\{\tau_n > t\}$,  
  \begin{align*}
    \E_\P \br{|S_{t+1}| \given \sigalg F_t} =  \lim_{k \uparrow \infty} \E_\P \br{|S_{t+1}| \wedge k \given \sigalg F_t} = 
    \lim_{k \uparrow \infty} \E_\P \br{|S_{t+1}^{\tau_n}| \wedge k \given \sigalg F_t} = \E_\P \br{|S_{t+1}^{\tau_n} |\given \sigalg F_t}  < \infty.
  \end{align*} 
  Since $\lim_{n \uparrow \infty} \tau_n = \infty$, we get $\E_\P
  \br{|S_{t+1}| \given \sigalg F_t}  < \infty$.  

  The next step we only argue for the case $d=1$, for sake of
  notation, but the general case follows in the same manner. 
  As above, again for fixed $n \in \N$, on the event
  $\{\tau_n > t\}$,  we get
  \begin{multline*}
    \E_\P \br{S_{t+1} \given \sigalg F_t} =  \lim_{k \uparrow \infty}
    \left( \E_\P \br{S_{t+1}^+ \wedge k \given \sigalg F_t} -  \E_\P
      \br{S_{t+1}^- \wedge k \given \sigalg F_t} \right) \\
    = 
    \lim_{k \uparrow \infty} \E_\P\br{(S_{t+1}^{\tau_n} \wedge k) \vee (-k) \given \sigalg F_t} = S_t.
  \end{multline*}
  Thanks again to  $\lim_{n \uparrow \infty} \tau_n = \infty$, the assertion follows.
\end{proof}

\begin{example}
  Assume that $(\Omega, \sigalg F, \P)$ supports two independent
  random variables $U$ and $\theta$ such that $U$ is uniformly
  distributed on $[0,1]$, and $\P[\theta = -1]  = \sfrac{1}{2} =
  \P[\theta = 1]$.  Moreover, let us assume that $\sigalg F_0 =
  \{\emptyset, \Omega\}$, $\sigalg F_1 = \sigma(U)$, and $\sigalg F_t
  = \sigma(U, \theta)$ for all $t \in \N \setminus \{1\}$. 
  Then the  stochastic process $S = (S_t)_{t \in \N_0}$, given by $S_t
  = \sfrac{\theta }{ U} \indicator{t \geq 2}$ is easily seen to be a
  $\P$--generalized martingale and a $\P$--local martingale with
  localization sequence $(\tau_n)_{n \in \N}$ given by  
  \begin{align*} 
    \tau_n = 1 \times \indicator{\{\sfrac{1}{U} >n\}} + \infty \times
    \indicator{\{\sfrac{1}{U} \leq n\}}.  
  \end{align*}
  However, we have $\E_\P[|S_2|] = \E_\P[\sfrac{1}{U}] = \infty$;
  hence $S$ is not a $\P$--martingale.  

  Now, consider the process $Z = (Z_t)_{t \in \N_0}$, given by 
  $
  Z_t = \indicator{t  = 0 } +  2 U  \indicator{t  \geq 1 }.
  $
  A simple computation shows that $Z$ is a strictly positive
  $\P$--uniformly integrable martingale. Moreover, since $Z_t S_t =
  2 \theta \indicator{t \geq 2}$, we have $\E_{\P}[Z_t |S_t|] \leq 2$
  for all $t \in \N_0$ and $Z S$ is a $\P$--martingale. 
  If we require the Radon-Nikodym to be bounded by a constant
  $1+\varepsilon \in (1, 2]$, we could consider $\widehat Z =
  (\widehat Z_t)_{t \in \N_0}$  with $\widehat Z_t = \indicator{t  = 0
  } +  \sfrac{(U \wedge \varepsilon) }{(\varepsilon -
    \sfrac{\varepsilon^2}{2})}  \indicator{t  \geq 1 }$. This
  illustrates the validity of Theorem~\ref{T:1} in the context of this
  example. 

  To see a difficulty in proving Theorem~\ref{T:1}, let us consider a
  local martingale $S' = (S'_t)_{t \in \N_0}$ with two jumps instead
  of one; for example, let us define 
  \begin{align*}
    S'_t =  \left( \indicator{\{U >  \sfrac{1}{2}\}} -  \indicator{\{U
    <  \sfrac{1}{2}\}} \right) \indicator{t \geq 1} +  \frac{\theta }{
    U} \indicator{t \geq 2}. 
  \end{align*}
  Again, it is simple to see that this specification makes $S'$ indeed
  a $\P$--local and $\P$--generalized martingale. However, now we have
  $\E_\P[Z_1 S_1'] = \sfrac{1}{2} \neq 0$; hence $Z S'$ is not a
  $\P$--martingale. Similarly, neither is $\widehat Z
  S'$. Nevertheless, as Theorem~\ref{T:1} states, there exists a
  uniformly integrable $\P$--martingale $Z'$ such that $Z' S'$ is a
  $\P$--martingale. 
\end{example}

More details on the previous example are provided in \cite{Ruf:LN_M}.

\section{Proof of Theorem~\ref{T:1}}  \label{S:proof}
In this section, we shall provide the proof of this note's main result. Its overall structure resembles
Theorem~1.3 in \cite{Prokaj:Rasonyi:2010}. The main novelty lies in Lemma~\ref{lem:1}, where the ideas of \cite{Rogers:1994} are adapted to obtain an equivalent martingale measure together with the required integrability condition (see Lemmata~\ref{L:1a} and \ref{L:2a}).  In contrast, the construction of the equivalent martingale measure in \cite{Prokaj:Rasonyi:2010}   is based on \cite{Dalang:Morton:Willinger}.

\begin{lemma} \label{lem:1}
  Let $\Qu$ denote some probability measure on
  $(\Omega, {\sigalg F})$, let ${\sigalg G}, {\sigalg H}$ be sigma algebras
  with ${\sigalg G} \subset {\sigalg H} \subset {\sigalg F}$,  let $W$ denote
  a  ${\sigalg H}$--measurable $d$-dimensional random vector with 
  \begin{equation}\label{eq:w}
    \E_{\Qu}\br{\abs{W}\given \sigalg G}< \infty 
    \qquad\text{and}\qquad 
    \E_\Qu\br{W\given{\sigalg G}} = 0.
  \end{equation}
  Suppose that $(\alpha_k)_{k \in \N}$ is a bounded family of
  ${\sigalg H}$--measurable random 
  variables with $\lim_{k \uparrow \infty} \alpha_k = 1$.
  Then for any $\varepsilon>0$ there
  exists a family $(V_k)_{k \in \N}$ of random variables such that 
  \begin{enumerate}[label={\rm(\roman{*})}, ref={\rm(\roman{*})}]
  \item\label{L:n:1} $V_k$ is $\sigalg H$--measurable and takes values in
    $(1-\varepsilon,1)$ for each $k \in \N$;
  \item\label{L:n:4} $\lim_{k \uparrow \infty} \indicator{\{\E_\Qu\br{V_k
      \alpha_k W\given\sigalg G} = 0\}}=1.$
  \end{enumerate}
\end{lemma}

We shall provide two proofs of this lemma, the first one applies only
to the case $d=1$, but avoids the technicalities necessary for the
general case. 
\begin{proof}[Proof of Lemma~\ref{lem:1} in the one-dimensional case]
  With the convention $\sfrac{0}{0} := 1$, define, for each $k
  \in \N$, the random variable
  \begin{align*}
    C_k = \frac{\E_{\Qu}\br{\alpha_k W^+ \given \sigalg G}}
    {\E_{\Qu}\br{\alpha_k W^-  \given \sigalg G}} 
  \end{align*} 
  and note that
  \begin{align*}
    \lim_{k \uparrow \infty} |C_k - 1| =  
    \left|\frac{\E_{\Qu}\br{W^+  \given\sigalg G} }{\E_{\Qu}\br{W^-\given\sigalg G}}  - 1\right| =  
    \frac{1}{\E_{\Qu}\br{W^-  \given \sigalg G}} 
    \left| \E_{\Qu}\br{W^+  \given\sigalg G} - \E_{\Qu}\br{W^-\given\sigalg G}
    \right| 
    = 0. 
  \end{align*}
  Next, set
  \begin{align*}
    V_k = (1-\varepsilon) \vee \left( \indicator{\{W \geq 0\}}  (1
    \wedge C_k^{-1}) + \indicator{\{W < 0\}}  (1 \wedge C_k) \right), 
  \end{align*}
  and note that on the event $\{1-\varepsilon \leq C_k \leq
  \sfrac{1}{(1-\varepsilon)}\} \in \sigalg G$ we indeed have 
  $\E_\Qu\br{V_k \alpha_k W\given\sigalg G} = 0$, which concludes the proof.
\end{proof}

\begin{proof}[Proof of Lemma~\ref{lem:1} in the general case]
  The proof is similar to the proof of the Dalang--Morton--Willinger theorem
  based on utility maximisation, see \cite{Rogers:1994} and
  \citet[Section 6.6]{DS_mono} for detailed exposition. 
  But instead of using the exponential utility, we choose a strictly
  convex function (the negative of the utility) which is smooth and
  whose derivative takes values in $(1-\varepsilon,1)$. 
  Indeed, in what follows we
  fix the  convex function 
  \begin{displaymath}
    f(a)  = 
    a\zjel*{1+\frac{\varepsilon}{\pi}\zjel*{\arctan(a)-\frac\pi2}}, 
    \qquad a \in \R.     
  \end{displaymath}
  Then $f$ is smooth and a direct computation shows that $f$ is convex
  with derivative $f'$ taking values in the interval $(1-\varepsilon,1)$. 

  We formulated the statement with generalized conditional
  expectations. However, 
  changing the probability appropriately 
  with a ${\sigalg G}$--measurable density  we can assume,
  without loss of generality, that $W \in \Lp1(\Qu)$. Indeed, the
  probability measure $\Qu'$, given by  
  \begin{displaymath}
    \frac{\d\Qu'}{\d \Qu} =\frac{e^{-\E_\Qu\br{\abs{W}\given{\sigalg
            G}}}}{\E_\Qu\br{e^{-\E_\Qu\br{\abs{W}\given{\sigalg G}}}}}, 
  \end{displaymath}
  satisfies that  $W \in \Lp1(\Qu')$. Moreover, the (generalized) conditional
  expectations with respect to ${\sigalg G}$ are the same under $\Qu$ and
  $\Qu'$. Hence, in what follows, we assume that $\abs{W}$ is an
  integrable random variable.

  For $W$ there is a maximal ${\sigalg G}$--measurable orthogonal projection
  $R$ of $\R^d$ such that $RW=0$ almost surely. The maximality
  of $R$ means that for 
  any $\sigalg G$--measurable vector variable $U$ which is orthogonal
  to $W$ almost surely we have $RU=U$. We shall use this property at
  the end of this proof, such that on the event $\event{RU\neq U}$ the
  scalar product $W\cdot U$ is non-zero with zero conditional mean so
  its conditional law is non-degenerate.
  The idea behind
  the construction of $R$ is to consider the space of ${\sigalg
    G}$--measurable vector variables orthogonal to $W$ almost surely,
  and ``take an orthonormal basis over each $\omega\in\Omega$'' in a
  ${\sigalg G}$--measurable way.   For details of the proof, see
  Proposition~2.4 in  
  \cite{Rogers:1994} or  Section~6.2 in \cite{DS_mono}. The
  orthocomplement of the  
  range of $R$ is called the predictable range of $W$.

  Let $B$ now  denote the $d$--dimensional
  Euclidean unit ball and set $\alpha_\infty = 1$. For each $k \in \N
  \cup \{\infty\}$,  
  consider the random function (or field) $h_k$ over $B$, defined by the formula
  \begin{displaymath}
    \begin{aligned}
      h_k(u,\cdot)
      &=h_k(u)=
      \E_\Qu\br{f(\alpha_k W \cdot u)\given{\sigalg G}}+\frac12\abs{Ru}^2 
      \qquad \text{for all $u\in B$}.   
    \end{aligned}
  \end{displaymath}
  Since $f$ is continuous,  for each $k \in \N \cup \{\infty\}$,
  $h_k$ has a version that is continuous in $u$ for each $\omega \in \Omega$;
  see Lemma~\ref{lem:cont} below.
  Then for each compact subset $C$ of $B$ and each $k \in \N \cup
  \{\infty\}$ there is a ${\sigalg G}$--measurable 
  random vector $U^C_k$ taking values in $C$ such that 
  $h_k(U^C_k)=\min_{u\in C} h_k(u)$.
  This is a kind of measurable
  selection; 
  for sake of completeness we give
  an elementary proof below in 
  Lemma~\ref{lem:ms}. 

  Next, for each $k \in \N$, let $U_k$ be  a ${\sigalg G}$--measurable
  minimiser of  $h_{k}$ in the unit ball $B$ and define 
  \begin{displaymath}
    V_k=f'(\alpha_k W\cdot U_k).
  \end{displaymath}
  With this definition, \ref{L:n:1} follows directly.
  For \ref{L:n:4} we prove below that
  \begin{align}
    \label{it:1}
    \E_\Qu\br{V_k \alpha_k W\given {\sigalg G}}+ R U_k
    &=0,\qquad
      \text{on $\event*{\abs{U_k}<1}$}, \quad k \in \N;
    \\
    \label{it:2}
    \lim_{k \uparrow \infty} U_k
    &= 0,\qquad\text{almost surely.}
  \end{align}
  Then, on the event $\{\abs{U_k}<1\}$,  \eqref{it:1} and the
  ${\sigalg G}$--measurability of $R$ yield 
  \begin{displaymath}
    \abs*{\E_\Qu\br{V_k \alpha_k W \given {\sigalg G}}}^2=
    -\E_\Qu\br{V_k \alpha_k W\given{\sigalg G}}\cdot RU_k = 
    -\E_\Qu\br{V_k \alpha_k RW\given{\sigalg G}}\cdot U_k =0,
  \end{displaymath}
  giving us \ref{L:n:4}.

  \bigskip

  Thus, in order to complete the proof it suffices to argue
  \eqref{it:1}--\eqref{it:2}. For \eqref{it:1}, note
  that   $h_k$ is continuously
  differentiable
  almost surely for each $k \in \N$, see Lemma~\ref{lem:C1} below;
  morever,  
  its derivative at the minimum point $U_k$, which equals the
  left-hand side of \eqref{it:1},  
  must be zero  when $U_k$ is inside the ball $B$. 

  For \eqref{it:2} observe that $h_\infty$ has a unique minimiser
  over $B$ which is the zero vector. To see this, observe that
  \begin{displaymath}
    h_\infty(u)=
    \E_\Qu\br{f(W\cdot(I-R) u)\given {\sigalg G}}+\frac12\abs{Ru}^2,
  \end{displaymath}
  where $I$ denotes the $d$-dimensional identity matrix.
  So to see that the zero
  vector is the unique minimiser it is enough to show that
  $\inf_{\abs{u}\geq \delta} h_\infty(u)>0 = h_\infty(0)$ almost
  surely for any $\delta\in(0,1]$. 
  Let $U$ be a ${\sigalg G}$--measurable minimiser of $h_\infty$ over
  $\set{u}{\abs{u}\in[\delta,1]}$. Then
  \begin{align*}
    \E_\Qu\br{f(W\cdot (I-R)U)\given{\sigalg G}}
    &>0,\qquad\text{on $\{(I-R) U\neq0\}$};\\
    \abs{RU}^2\geq \delta^2&>0,\qquad\text{on $\{(I-R)U=0\}$}.
  \end{align*}
  The first part follows from the strict convexity of $f$ in
  conjunction with Jensen's inequality, taking into 
  account that  $\E_\Qu\br{W\given{\sigalg G}}=0$ and that $W\cdot(I-R)U$
  has non-trivial conditional law  on $\{(I-R) U\neq0\}$ by the
  maximality of $R$.  Whence
  $\inf_{\abs{u}\geq\delta}h_\infty(u)>0=h_\infty(0)$, as required.  
  
  Finally, as $\lim_{k \uparrow \infty} \alpha_k = 1$ and $f$ is Lipschitz 
  continuous we have
  \begin{displaymath}
    \lim_{k \uparrow \infty} \sup_{u\in B}\abs{h_k(u)-h_\infty(u)}  =   \lim_{k \uparrow \infty} \sup_{u\in B \cap \Q^d}\abs{h_k(u)-h_\infty(u)} = 0
    \qquad\text{almost surely}.
  \end{displaymath}
  Hence, any ${\sigalg G}$--measurable sequence $(U_k)_{k \in \N}$ of
  minimisers of 
  $h_k$ converges to zero, the unique minimiser of $h_\infty$,
  almost surely. This shows \eqref{it:2} and completes the proof. 
\end{proof}

\begin{lemma} \label{L:1a}
  Let $\Qu$ denote some probability measure on $(\Omega, \sigalg
  F)$,
  let $\sigalg G, \sigalg H$ be sigma algebras with $\sigalg G \subset
  \sigalg H \subset \sigalg F$,  let $Y$ denote a one-dimensional
  random variable with $Y \geq 0$ and $\E_\Qu\br{Y \given \sigalg H}  <
  \infty$, and let $W$ denote  a  ${\sigalg H}$--measurable
  $d$-dimensional random vector such that \eqref{eq:w} holds. 
  Then, for any $\varepsilon> 0$, there exists a random
  variable $z$ such that 
  \begin{enumerate}[label={\rm(\roman{*})}, ref={\rm(\roman{*})}]
  \item\label{L:1:1} $z$ is $\sigalg H$--measurable and takes values
    in $(0,1 + \varepsilon)$; 
  \item\label{L:1:3} $\Qu[z < 1 - \varepsilon] < \varepsilon$;
  \item\label{L:1:4} $\E_{\Qu}\br{z \given \sigalg G} = 1$;
  \item\label{L:1:5} $\E_{\Qu}\br{z W\given \sigalg G} = 0$;
  \item\label{L:1:6} $\E_{\Qu}\br{z Y\given \sigalg G} < \infty$.
  \end{enumerate}
\end{lemma}

\begin{proof}
  For each $k \in \N$, define  the $(0,1]$--valued,  $\sigalg
  H$--measurable random variable 
  \begin{align*}
      \alpha_k = \indicator{\{ \E_\Qu\br{Y \given \sigalg H}  \leq k\}}
    + 
    \frac{1}{\E_\Qu\br{Y \given \sigalg H} }  
    \indicator{\{ \E_\Qu\br{Y \given \sigalg H}  > k\}}
  \end{align*}
  and note that $\lim_{k \uparrow \infty} \alpha_k =
  1$. Lemma~\ref{lem:1} now yields the existence of a family
  $(V_k)_{k\in \N}$ of $\sigalg H$--measurable random variables such
  that   $V_k \in (\sfrac{1}{(1+\varepsilon/2)},1)$  and 
  $\lim_{k\uparrow\infty}\indicator{\{\E_\Qu\br{V_k\alpha_k
      W\given\sigalg G}=0\}}=1$.
  Note that this yields a $\sigalg G$--measurable random variable $K$,
  taking values in $\N$, such that 
  $\E_\Qu\br{V_K\alpha_K W\given\sigalg G} = 0$, 
  $\E_\Qu\br{V_K\alpha_K\given\sigalg G}>\sfrac{1}{(1+\varepsilon)}$, 
  and $\Qu\br{\E_\Qu\br{Y \given \sigalg H}>K}<\varepsilon$. Setting now 
  \begin{align*}
    z= \frac{V_K \alpha_K} {\E_{\Qu}\br{V_K \alpha_K \given \sigalg G}}
  \end{align*}	
  yields a random variable with the claimed properties.
\end{proof}

\begin{lemma}
  \label{L:2a}
 	Fix $n \in \N_0$, let  $\Qu$ denote some probability measure
  on $(\Omega, \sigalg F)$ such that $S$ is a $\Qu$--local
  martingale, and let $Y$ denote a one-dimensional
  random variable with $Y \geq 0$ and
  $\E_\Qu\br{Y \given \sigalg F_n}  < \infty$. 
  Then,  for each $\varepsilon > 0$,  there exists a probability measure $\Qu'$, equivalent to $\Qu$,
  with density $Z^{(n)}=\sfrac{\d\Qu'}{\d\Qu}$ such that
  \begin{enumerate}[label={\rm(\roman{*})}, ref={\rm(\roman{*})}]
  \item\label{L:2a:2} $Z^{(n)} \in (0,1 + \varepsilon)$;
  \item\label{L:2a:3} $\Qu[Z^{(n)} < 1 - \varepsilon] < \varepsilon$;
  \item\label{L:2a:5} $ S$ is a $\Qu'$--local martingale;
  \item\label{L:2a:6} $\E_{\Qu'}[Y] < \infty$.
  \end{enumerate}
\end{lemma}
\begin{proof}
  In this proof, we use the convention $\sigalg F_{-1} = \{\emptyset,
  \Omega\}$ and $\Delta S_0 = 0$. 
  Set $\widetilde\varepsilon>0$ be sufficiently small such that
  \begin{displaymath}
    (n+1)\widetilde\varepsilon\leq\varepsilon,\qquad
    (1+\widetilde\varepsilon)^{n+1}\leq 1+\varepsilon,\qquad
    (1-\widetilde\varepsilon)^{n+1}\geq 1-\varepsilon.
  \end{displaymath}
  We shall construct a sequence $(z_0, \cdots, z_n)$ iteratively starting  with
  $z_n$  and proceeding backward until $z_0$ 
  such that for each $t=0,1,\dots,n$,
  \begin{equation}\label{eq:t1}
    z_t\leq 1+\tilde\varepsilon,\quad
    \Qu\br{z_t<1-\tilde\varepsilon}<\tilde\varepsilon,\quad
    \E_{\Qu}\br{z_t\given \sigalg F_{t-1}}=1,\quad
    \E_{\Qu}\br{z_t\Delta S_t\given \sigalg F_{t-1}}=0,
  \end{equation}
  and
  \begin{equation}
    \label{eq:t2}
    \E_{\Qu}\br{Y\prod_{i=t}^{n}z_i\given\sigalg F_{t-1}}<\infty.      
  \end{equation}
  
  For $t=n$ we apply Lemma~\ref{L:1a}
  with $\varepsilon$  replaced
  by $\widetilde \varepsilon$
  and with $\sigalg G = \sigalg F_{n-1}$, $\sigalg H = \sigalg F_n$,
  and $W = \Delta S_n$. We have $\E\br{Y|\sigalg H}<\infty$ by assumption and
  $\E_{\Qu}\br{|W|\given\sigalg G} < \infty$ and
  $\E_{\Qu}\br{W\given\sigalg G} = 0$ by Proposition~\ref{P:1}. Hence,
  Lemma~\ref{L:1a} provides us an appropriate $z_n$ satisfying
  \eqref{eq:t1} and \eqref{eq:t2} for $t=n$.

  For $0\leq t <n$ assume that we have random
  variables $z_{t+1}, \cdots, z_n$ satisfying
  \eqref{eq:t1} and \eqref{eq:t2}, 
  in particular,
  $\E_{\Qu}\br{Y \prod_{i = t+1}^n z_i\given \sigalg F_t} <
  \infty$. We now
  obtain a random variable $z_t$ by again applying Lemma~\ref{L:1a},
  with $\varepsilon$ replaced by $\widetilde \varepsilon$
  and with $\sigalg G = \sigalg
  F_{t-1}$, $\sigalg H = \sigalg F_t$, $W = \Delta S_t$, and $Y$
  replaced by $Y\prod_{i = t+1}^n z_i $.   
  
  With the family $(z_0, \cdots, z_n)$ now given, let us define 
  $Z^{(n)} = \prod_{i = 0}^ n z_i$ and $\Qu'$ by
  $\sfrac{\d\Qu'}{\d\Qu}=Z^{(n)}$. With this definition of $Z^{(n)}$ 
  \ref{L:2a:2},\ref{L:2a:3}, and \ref{L:2a:6} are clear by the choice
  of $\tilde\varepsilon$. To argue that 
  $S$ is a $\Qu'$--local martingale, let $\tau$ be an $(\sigalg
  F_t)_{t\geq0}$ stopping time such that the stopped process $S^\tau$
  is a martingale. Then $\Delta S_t\indicator{\tau\geq t}$ is $\Qu'$ integrable
  random vector as $Z^{(n)}$ is bounded from above. Moreover,  Bayes' rule yields
  \begin{displaymath}
    \E_{\Qu'}\br{\Delta S_t \indicator{\{\tau\geq t\}}\given\sigalg F_{t-1}}=
    \frac{\E_{\Qu}\br{Z^{(n)}\Delta S_t \indicator{\{\tau\geq t\}}\given\sigalg F_{t-1}}}{\E_{\Qu}\br{Z^{(n)}\given\sigalg F_{t-1}}}=
  \indicator{\{\tau\geq t\}}\E_{\Qu}\br{z_t\Delta S_t\given\sigalg F_{t-1}}=0.
  \end{displaymath}
  So any sequence of stopping times that localizes $S$ under $\Qu$
  also localizes it under $\Qu'$. This shows \ref{L:2a:5}; hence the lemma is proven. 
\end{proof}

\begin{proof}[Proof of Theorem~\ref{T:1}]
    We inductively construct a sequence $(\Qu^{(n)})_{n\in\N_0}$ of probability measures, equivalent to $\P$,
     and a sequence $(\varepsilon^{(n)})_{n\in\N_0}$ of positive reals
     using Lemma~\ref{L:2a}.
    To start, set $\Qu^{(-1)} = \P$.
    Now, fix
    $n \in \N_0$ for the moment and suppose that we have $\Qu^{(n-1)}$ and
    $(\varepsilon^{(m)})_{0\leq m<n}$ such that 
    $\prod_{m=0}^{n-1}(1+\varepsilon^{(m)})<1+\varepsilon$.
    Choose $\varepsilon^{(n)}$ to be
    sufficiently small such that
    $\prod_{m=0}^{n}(1+\varepsilon^{(m)})<1+\varepsilon$,  and
    for any $A\in\sigalg F$ with $\Qu^{(n-1)}\br{A}\leq \varepsilon^{(n)}$
    we have $\P\br{A}<2^{-n}$. Then apply Lemma~\ref{L:2a} with $\varepsilon$
    replaced by $\varepsilon^{(n)}$,
    and with $\Qu = \Qu^{(n-1)}$ and $Y = \e^{|S_n|}$ to obtain a
    probability measure 
    $\Qu^{(n)}$ with density 
    $Z^{(n)}$, that is
    $\d\Qu^{(n)}=Z^{(n)}\d\Qu^{(n-1)}=\zjel{\prod_{m=0}^nZ^{(m)}}\d\P$. 

  Due to the fact
    \begin{displaymath}
      \P\br*{\abs{1-Z^{(n)}}>\varepsilon^{(n)}}\leq2^{-n}\quad\text{as}\quad
      \Qu^{(n-1)}\br*{\abs{1-Z^{(n)}}>\varepsilon^{(n)}}\leq\varepsilon^{(n)},
    \end{displaymath}
     the Borel-Cantelli lemma yields $\sum_{n \in \N_0} \abs{1-Z^{(n)}} < \infty$; hence the infinite product $Z_{\infty}=\prod_{n=0}^\infty Z^{(n)}$
    converges and is positive $\P$--almost surely. It is clear that
    $Z_{\infty}\leq 1+\varepsilon$.

  We define the probability measure $\Qu$ by $\sfrac{\d\Qu}{\d
    \P}=Z_{\infty}$ and denote the corresponding density process by
  $Z_t=\E_{\P}\br{Z_\infty\given\sigalg F_t}$, for each $t \in
  \N_0$. 
As $\prod_{m>t}Z^{(m)}<1+\varepsilon$ 
  we have  $\Qu\leq (1+\varepsilon)\Qu^{(t)}$ 
  and as a result 
  \begin{displaymath}
    \E_{\P}\br*{Z_te^{\abs{S_t}}}=\E_{\Qu}\br*{e^{\abs{S_t}}}\leq
    (1+\varepsilon)\E_{\Qu^{(t)}}\br*{e^{\abs{S_t}}}<\infty 
  \end{displaymath}
  by the choice of $\Qu^{(t)}$; hence $\E_{\P}\br{Z_t\abs{S_t}^p}<\infty$ for all $t,p \in \N_0$.
  
  It remains to argue that $ZS$ is a $\P$--martingale or, equivalently, that $S$ is a
 $\Qu$--martingale.
Since we already have established $\E_{\Qu}\br{\abs{S_t}}<\infty$ for all $t \in \N_0$, it suffices to fix $t \in \N$ and to prove that $\E_{\Qu}\br{S_t\given\sigalg F_{t-1}} = S_{t-1}$. To this end, recall that $S$ is a $\Qu^{(n)}$--local martingale for each $n \in \N_0$ by Lemma~\ref{L:2a}\ref{L:2a:5} and note that
 dominated convergence, Bayes formula, and Proposition~\ref{P:1} yield
  \begin{align*}
    \E_{\Qu}\br{S_t\given\sigalg F_{t-1}}Z_{t-1}
    & =
    \E_{\P}\br{S_tZ_\infty\given\sigalg F_{t-1}} 
     =
    \lim_{n\uparrow\infty}\E_{\P}\br*{S_t\prod_{m=0}^nZ^{(m)}\given\sigalg F_{t-1}}\\
    & =
    \lim_{n\uparrow\infty}\E_{\Qu^{(n)}}\br{S_t\given\sigalg F_{t-1}}
      \left.\frac{\d\Qu^{(n)}}{\d\P}\right|_{\sigalg F_{t-1}} = S_{t-1}
      \lim_{n\uparrow\infty}\E_{\P}\br*{\prod_{m=0}^{n}Z^{(m)}\given\sigalg F_{t-1}}\\
      &=S_{t-1}Z_{t-1}.
  \end{align*}
This
  completes the proof. 
\end{proof}

\appendix
\section{Appendix}

In this appendix, we provide some  
measurability results necessary for
the proof of Lemma~\ref{lem:1}.  We write  $C(K)$ for the space of continuous functions over some metric space $(K,m)$ and equip $C(K)$ with the supremum norm.  

When a random variable takes values in an abstract measurable space we
call it a random element from that space. In all cases below, the measurable space is a metric space
equipped with its Borel $\sigma$-algebra, the $\sigma$-algebra
generated by the open sets. In particular, $\xi$ is a random element
from $C(K)$ if and only if $\xi(u)$ is a random variable for each $u$
and $u\mapsto\xi(u,\omega)$ is continuous for each $\omega\in\Omega$.

\begin{lemma}\label{lem:cont}
  Let  ${\sigalg G}$ be a sigma algebra
  with ${\sigalg G}  \subset {\sigalg F}$ and let $\xi$ be a random
  element in $C(K)$,  where $(K,m)$ is a compact metric space.
  Suppose that $\E_\P[\sup_{u\in K} \abs{\xi(u)}]<\infty$ and let
  $\eta(u)=\E_\P\br{\xi(u)\given{\sigalg G}}$ for all $u\in K$. Then 
  $(\eta(u))_{u\in K}$ has a continuous modification.
\end{lemma}
\begin{proof}
  Let $D$ be a countable dense subset of $K$. We show that there is
  $\Omega'\in{\sigalg G}$ with full probability such that
  $(\eta(u))_{u\in D}$ is uniformly continuous over $D$ on $\Omega'$.
  Then we can define
  \begin{displaymath}
    \tilde{ \eta}(u)=
    \begin{cases}
      \lim\limits_{u_n\to u\atop u_n\in D} \eta(u_n)&\text{on $\Omega'$},\\
      0&\text{otherwise}.
    \end{cases}
  \end{displaymath}
  It is a routine exercise to check that $\tilde{\eta}$ is well defined and a
  continuous modification of $\eta$.

  One way to get $\Omega'$ is the following. 
  Let $\mu$ be the modulus of continuity of $\xi$, that is,
  \begin{displaymath}
    \mu(\delta)=
    \sup_{u,u'\in K, \, m(u,u') \leq \delta} \abs{\xi(u)-\xi(u')}, 
    \qquad \delta > 0. 
  \end{displaymath}
  Obviously $\mu(\delta)\to0$
  everywhere as $\delta\downarrow 0$. 
  Dominated convergence, in conjunction with the bound $\mu \leq
  2\sup_{u\in K} \abs{\xi(u)}$, yields
  $\tilde{\mu}(\delta)=\E\br{\mu(\delta)\given{\sigalg G}}\to 0$ as
  $\delta\downarrow 0$ 
  almost surely. 
  Now define
  \begin{align*}
    \Omega' = \event*{\lim_{n\uparrow\infty}
    \tilde{\mu}\left(\frac{1}{n}\right)=0} \cap 
    \zjel*{ \bigcap_{n \in \N}\,\, \bigcap_{u,u'\in D, \, m(u,u')\leq
    \sfrac{1}{n}} 
    \event*{\abs{\eta(u)-\eta(u')}\leq   \tilde{\mu}\zjel*{\frac{1}{n}}}}.
  \end{align*} Clearly $\Omega'$ has full probability and the claim is proved.
\end{proof}

In the setting of Lemma~\ref{lem:cont}
when $K\subset\R^d$ and $\xi$ is  a random element in $C^1(K)$ then
under mild conditions $\eta(u)=\E\br{\xi(u)\given\sigalg G}$ has a
version taking values in $C^1(K)$. This is the content of the next
lemma. Recall that a function $f$ defined on $K$ belongs to $C^1(K)$
if $f$ is continuous and 
there is a continuous $\R^d$--valued function on $K$ which agrees with
the gradient $f'$ of $f$ in the interior of $K$.

\begin{lemma}\label{lem:C1}
  Let  ${\sigalg G}$ be a sigma algebra
  with ${\sigalg G}  \subset {\sigalg F}$ and let $\xi$ be a random
  element in $C^1(K)$,  where $K\subset\R^d$ is a compact 
  subset set.
  Suppose that $$\E_\P\left[\sup_{u\in K} \abs{\xi(u)}\right] + \E_\P\left[\sup_{u\in K} \abs{\xi'(u)}\right]<\infty$$ and let
  $\eta(u)=\E_\P\br{\xi(u)\given{\sigalg G}}$ for all $u\in K$. Then 
  $(\eta(u))_{u\in K}$ has a version taking values in $C^1(K)$ and
  the continuous version of $(\E\br{\xi'(u)\given\sigalg G})_{u \in K}$ gives the gradient of $\eta$ almost surely.
\end{lemma}
\begin{proof}By Lemma~\ref{lem:cont} both 
  $\eta(u)=\E\br{\xi(u)\given\sigalg G}$ and 
  $\eta'(u)=\E\br{\xi'(u)\given\sigalg G}$ have continuous
  versions. We prove that, apart from a null set, $\eta'$ is indeed the
  gradient of $\eta$. To this end, let $D$ be a countable dense subset of the
  interior of $K$ and denote by $I(a,b)$ a directed segment going
  from $a$ to $b$, for each $a,b \in K$.  Then, by assumption,   for $a,b\in D$,
  with $I(a,b)\subset\interior K$ we get
  \begin{displaymath}
    \eta(b)-\eta(a)=\E\br{\xi(a)-\xi(b)\given\sigalg G}=
    \E\br*{\int_{I(a,b)} \xi'(u)\d u\given\sigalg G}=
    \int_{I(a,b)}\eta'(u)\d u, \quad\text{ almost surely.}
  \end{displaymath}
  Hence, there exists an event $\Omega'\in\sigalg G$ with $\P[\Omega'] = 1$ such that
  \begin{displaymath}
    \eta(b,\omega)-\eta(a,\omega)=\int_{I(a,b)}\eta'(u,\omega)du,\quad
    \text{for all $a,b\in D$, with $I(a,b)\subset\interior K$ and
      $\omega\in\Omega'$}.
  \end{displaymath}
  By continuity this identity extends to all $a,b\in\interior K$ with
  $I(a,b)\subset \interior K$ on $\Omega'$. Using again the continuity
  of $\eta'(.,\omega)$ yields that $\eta'$ is indeed the gradient of
  $\eta$ on $\Omega'$. 
\end{proof}

\begin{lemma}\label{lem:ms}
  Let $(K,m)$ be a compact metric space and $\eta$  a random element
  in $C(K)$. Then there is a measurable minimiser of $\eta$, that is, a
  random   element $U$ in $K$ such that $\eta(U)=\min_{u\in K} \eta(u)$.
\end{lemma}
\begin{proof}
  To shorten the notation, for each $x \in K$ and $\delta \geq 0$, let 
  \begin{displaymath}
    B(x,\delta)=\set{u\in K}{ m(u,x)\leq \delta},\quad
    \eta(x,n)=\min\set{\eta(u)}{u\in B(x,2^{-n})}. 
  \end{displaymath}
  For each $n \in \N$ let $D_n$ be a finite $2^{-n}$-net in $K$; that
  is, $K \subset \bigcup_{x \in D_n}  B(x,2^{-n})$. 
  For
  each $n \in \N$ fix an order of the finite set $D_n$. 
  We shall use the fact  that for any closed set $F \subset K$ the minimum over
  $F$, that is, $\min_{u\in F}\eta(u)$, is a random variable.  This follows easily
since a continuous function on a metric space is Borel measurable, and
$C(K)\ni f\mapsto\inf_{u\in F} f(u)$ 
depends continuously on $f$, it is even Lipschitz continuous.    

  We construct a sequence  $(U_n)_{n \in \N}$ of random elements in $K$
  by recursion, such that 
  \begin{itemize}
  \item $\eta(U_n,n)=\min_{u\in K} \eta(u)$, and
  \item  $m(U_n,U_{n+1})\leq 2^{-n}+2^{-(n+1)}$.
  \end{itemize}
  Then $(U_n)_{n \in \N}$ has a limit $U$ which is a measurable
  minimiser of $\eta$ over $K$. To see that $U$ is indeed a minimiser,
  observe that for each $r>0$ there is an $n$ such that
  $B(U_n,2^{-n})\subset B(U,r)$, hence
  $$
  \min_{u\in K}\eta(u)\leq \min_{u\in B(U,r)}\eta(u)\leq
  \min_{u\in B(U_n,2^{-n})} \eta(u)=\eta(U_n,n)=\min_{u\in K} \eta(u).
  $$
  That is, the minimum of $\eta$ over the closed ball around $U$ with
  an arbitrary small positive radius $r$ agrees with the global
  minimum of $\eta$. 
  Letting $r\to0$ the continuity of $\eta$ yields that
  $\eta(U)=\min_{u\in K}\eta(u)$.  
  
  We now construct the sequence  $(U_n)_{n \in \N}$. For $n=1$ let $U_1$ be the first element in 
  \begin{displaymath}
    \set*{v\in D_1}{\eta(v,1)=\min_{u\in K}\eta(u)}.
  \end{displaymath}
  Since this set is not empty, $U_1$ is well defined. Moreover, $U_1$ takes
  values in the finite set $D_1=\smallset{v_1,\dots,v_k}$, and the
  levelset $\event{U_1=v_\ell}=A_\ell\setminus\cup_{i<\ell}A_i$, where
  $A_i=\event{\eta(v_i,1)=\min_{u\in K}\eta(u)}$, is obviously an
  event, as $\eta(v_i,1)$ and $\min_{u\in K}\eta(u)$ are random
  variables.  So $U_1$ is measurable, that is, a random 
  element from $K$. 

  If $U_1,\dots,U_n$ are defined for some $n \in \N$ set $U_{n+1}$ to
  be the first element in 
  \begin{displaymath}
    \set*{v\in D_{n+1}}{\eta(v,n+1)=\min_{u\in K}\eta(u),\,
      m(v,U_n)\leq  2^{-n}+2^{-(n+1)}}
  \end{displaymath}
  This set is not empty as 
  \begin{displaymath}
    B(U_n,2^{-n})\subset \bigcup_{v\in D_{n+1}\atop m(v,U_n)\leq 
      2^{-n}+2^{-(n+1)}} B(v,2^{-(n+1)}),
  \end{displaymath}
  so $U_{n+1}$ is well defined and its measurability is obtained
  similarly to that of 
  $U_1$. We conclude that
  the sequence with the above properties exists and its  limit is a
  measurable minimiser.
\end{proof}

\ACKNO{We thank Yuri Kabanov  for many helpful comments.}
\providecommand{\doi}[1]{doi:\href{https://doi.org/#1}{\nolinkurl{#1}}}

\end{document}